\documentclass{amsart}
\usepackage{bookmark}
\usepackage[utf8]{inputenc}

\usepackage{amssymb,comment}
\usepackage{mathrsfs}
\usepackage{stmaryrd}

\usepackage{tikz-cd}

\usepackage{mathabx}

\usepackage{lmodern}
\usepackage{csquotes}

\usepackage[T1]{fontenc}

\definecolor{darkgreen}{rgb}{0,0.5,0}
\definecolor{darkblue}{rgb}{0,0,0.8}
\definecolor{darkred}{rgb}{0.8,0,0}
\definecolor{lightblue}{rgb}{0,0.6,0.8}

\usepackage[english]{babel}
\usepackage{url}

\usepackage{multirow}

\usepackage{enumerate}
\usepackage{colonequals}
\usetikzlibrary{matrix, calc, arrows}

\DeclareFontFamily{U}{wncy}{}
\DeclareFontShape{U}{wncy}{m}{n}{<->wncyr10}{}
\DeclareSymbolFont{mcy}{U}{wncy}{m}{n}
\DeclareMathSymbol{\Sha}{\mathord}{mcy}{"58}

\usepackage{microtype}

\theoremstyle{plain}
\newtheorem{theorem}{Theorem}[subsection]
\newtheorem{lemma}[theorem]{Lemma}
\newtheorem{corollary}[theorem]{Corollary}
\newtheorem{proposition}[theorem]{Proposition}

\newtheorem{maintheorem}{Theorem}

\theoremstyle{definition}
\newtheorem{definition}[theorem]{Definition}

\newtheorem{example}[theorem]{Example}
\newtheorem{remark}[theorem]{Remark}

\newtheorem*{assumption*}{Assumption}
\newtheorem*{claim*}{Claim}

\theoremstyle{remark}

\usepackage{cleveref}
\crefname{theorem}{Theorem}{Theorems}
\crefname{lemma}{Lemma}{Lemmata}
\crefname{corollary}{Corollary}{Corollaries}
\crefname{proposition}{Proposition}{Propositions}
\crefname{definition}{Definition}{Definitions}
\crefname{conjecture}{Conjecture}{Conjectures}
\crefname{question}{Question}{Questions}
\crefname{example}{Example}{Examples}
\crefname{algorithm}{Algorithm}{Algorithms}
\crefname{remark}{Remark}{Remarks}
\crefname{assumption}{Assumption}{Assumptions}
\crefname{maintheorem}{Theorem}{Theorems}

\def\ol#1{\overline{#1}}
\def\wt#1{\widetilde{#1}}

\def\Alphabet{A,B,C,D,E,F,G,H,I,J,K,L,M,N,O,P,Q,R,S,T,U,V,W,X,Y,Z}
\def\alphabet{a,b,c,d,e,f,g,h,i,j,k,l,m,n,o,p,q,r,s,t,u,v,w,x,y,z}
\def\endpiece{xxx}
\def\makeAlphabet[#1]{\expandafter\makeA#1,xxx,}
\def\makealphabet[#1]{\expandafter\makea#1,xxx,}
\def\makeA#1,{\def\temp{#1}\ifx\temp\endpiece\else%
	\mkbb{#1}\mkfrak{#1}\mkbf{#1}\mkcal{#1}\mkscr{#1}\mkbs{#1}\expandafter\makeA\fi}%
\def\makea#1,{\def\temp{#1}\ifx\temp\endpiece\else\mkfrak{#1}\mkbf{#1}\mkbs{#1}\expandafter\makea\fi}%
\def\mkbb#1{\expandafter\def\csname bb#1\endcsname{\mathbb{#1}}}
\def\mkfrak#1{\expandafter\def\csname fr#1\endcsname{\mathfrak{#1}}}
\def\mkbf#1{\expandafter\def\csname b#1\endcsname{\mathbf{#1}}}
\def\mkcal#1{\expandafter\def\csname c#1\endcsname{\mathcal{#1}}}
\def\mkscr#1{\expandafter\def\csname s#1\endcsname{\mathscr{#1}}}
\def\mkbs#1{\expandafter\def\csname bs#1\endcsname{{\boldsymbol{#1}}}}
\def\makeop[#1]{\xmakeop#1,xxx,}
\def\mkop#1{\expandafter\def\csname #1\endcsname{{\operatorname{#1}}}} %
\def\xmakeop#1,{\def\temp{#1}\ifx\temp\endpiece\else\mkop{#1}\expandafter\xmakeop\fi}%
\def\makeup[#1]{\xmakeup#1,xxx,}
\def\mkup#1{\expandafter\def\csname #1\endcsname{{\mathrm{#1}\,}}} %
\def\xmakeup#1,{\def\temp{#1}\ifx\temp\endpiece\else\mkup{#1}\expandafter\xmakeup\fi}%
\makeAlphabet[\Alphabet]
\makealphabet[\alphabet]
\makeop[Hom,Tor,Ext,holim,hocolim,dim,Sel,GL,PGL,SL,H,ord,Sym,Gal,Ann,ind,Aut,End,cd,Frob,Gr,sp,Tot,sm,an,Pic,NS,Br,tors,Reg,ss,new,cts,ur,alg,pd,disc,cyc,rk,cork,ab,nr,tors,Iw,res,inf,Cl,Quot,corank,ord,length,proj,Hg,im,coker,Tam,loc,cores,Char,Fitt,free,Fil,MSD,Ind,fn,min]
\makeup[Spec,Proj,Spwf,Sp,Sh,Spf,Sch,id,dR,rig,cris,ur,div,ac,BK]

\newcommand{\F}{\mathbf{F}}

\newcommand{\fm}{\mathfrak{m}}

\newcommand{\fX}{\mathfrak{X}}

\renewcommand{\epsilon}{\varepsilon}
\renewcommand{\theta}{\vartheta}
\renewcommand{\phi}{\varphi}

\newcommand{\mathup}[1]{\text{\textup{#1}}}
\renewcommand{\H} {\ensuremath{\mathup{H}}}

\newcommand{\defeq}{\colonequals}

\newcommand{\iso}{\simeq}
\newcommand{\isom}{\cong}

\newtheorem*{conjecture*}{\textbf{Conjecture}}

\numberwithin{equation}{section}

\begin{document}
\title[Greenberg's conjecture at Eisenstein primes]{On Greenberg's conjecture for rational elliptic curves at Eisenstein primes}
\author{Zichao Lin}
\address{University of Massachusetts Amherst}
\email{zichaolin@umass.edu}
\author{Mulun Yin}
\address{Morningside Center of Mathematics; Academy of Mathematics and Systems Sciences, Chinese Academy of Science}
\email{mulunyin@amss.ac.cn}
\date{\today}


\begin{abstract}
    Let $E/\bQ$ be an elliptic curve, and let $p$ be an odd prime of ordinary reduction for $E$, and assume that $E$ admits a rational $p$-isogeny. In this paper we prove Greenberg's conjecture on the vanishing of algebraic Iwasawa $\mu$-invariants of the Selmer groups attached to $E$ over the cyclotomic $\bZ_p$-extension of the rational numbers. Via the Iwasawa Main Conjecture for elliptic curves at Eisenstein primes, we also confirm Stevens's conjecture on the behavior of analytic $\mu$-invariants.
\end{abstract}

\maketitle
\tableofcontents

\section*{Introduction}
\addtocontents{toc}{\protect\setcounter{tocdepth}{0}}
\subsection{Statement of the main results}

Let $E/\bQ$ be an elliptic curve and let $p$ be an odd prime of ordinary reduction for $E$. We say that $p$ is an \emph{Eisenstein} prime for $E$ if $E[p]$ is reducible as a $G_\bQ$-module, where $G_\bQ=\Gal(\ol \bQ/\bQ)$ is the absolute Galois group of $\bQ$ and $E[p]$ is the $p$-torsion of $E$. Equivalently, $p$ is an Eisenstein prime for $E$ if $E$ admits a rational $p$-isogeny.

Fix a prime number $p$. Let $\bQ_p$ be the $p$-adic numbers with integers $\bZ_p$ and residue field $\F$. Let $\Gamma\defeq\Gal(\bQ_\infty/\bQ)\iso (\bZ_p,+)$ be the Galois group of the cyclotomic $\bZ_p$-extension of $\bQ$, which is the unique Galois extension of $\bQ$ whose Galois group is isomorphic to the additive group of the $p$-adic integers. More generally, for a number field $K$, denote by $K_\infty\defeq K\bQ_\infty$ its cyclotomic $\bZ_p$-extension. Let $\Lambda\defeq\bZ_p\llbracket \Gamma\rrbracket$ be the associated Iwasawa algebra, which is usually identified with the formal power series ring $\bZ_p\llbracket T\rrbracket$ when a topological generator $\gamma\in\Gamma$ is sent to $1+T$. Attached to a rational elliptic curve $E$ for which $p$ is of ordinary reduction are the following two important arithmetic objects: \begin{itemize}
    \item a Selmer group $\Sel_\Gr(E/\bQ_\infty)$ \`a la Greenberg, whose Pontryagin dual $\fX_\Gr(E/\bQ_\infty)$ is a finitely generated torsion $\Lambda$-module,
    \item a $p$-adic $L$-function $\cL_E^\MSD$ \`a la Mazur--Swinnerton-Dyer, which is an element of $\Lambda$.
\end{itemize} 

When $p$ is a prime of multiplicative reduction, $\Sel_\Gr(E/\bQ_\infty)$ still makes sense and is $\Lambda$-cotorsion, and Mazur--Tate--Teitelbaum~\cite{MTT} have constructed a $p$-adic $L$-function generalizing $\cL_E^\MSD$. By an abuse of notation, we still denote the $p$-adic $L$-function by $\cL_E^\MSD$ even in the multiplicative reduction case.

When $E$ has ordinary reduction at $p$, the famous Iwasawa Main Conjecture asserts that the $\Lambda$-characteristic ideal $\Char_\Lambda(\fX_\Gr(E/\bQ_\infty))\subset\Lambda$ associated to $\Sel_\Gr(E/\bQ_\infty)$ is generated by the single power series $\cL_E^\MSD\in\Lambda$, and is very recently completely proved when $p$ is an odd good ordinary Eisenstein prime for $E$ in a series of works~\cite{GV00},~\cite{CGS} and~\cite{KY24}. 
Thus if we let $\cL_E^\Gr$ denote a generator of $\Char_\Lambda(\fX_\Gr(E/\bQ_\infty))$ (it is sometimes called the \emph{algebraic} $p$-adic $L$-function associated to $E$), then the Iwasawa Main Conjecture asserts that\[
\cL_E^\Gr=\cL_E^\MSD\text{ in }\Lambda
\]up to a $p$-adic unit.

In particular, the Iwasawa Main Conjecture predicts that the Iwasawa $\mu$-invariants  (that is, the maximal $p$-power $p^{\mu(\cL)}$ of an element $\cL\in\Lambda$ such that $\cL$ belongs to $p^{\mu(\cL)}\Lambda$) $\mu(\cL_E^\Gr)$ and $\mu(\cL_E^\MSD)$ should agree. In fact, these $\mu$-invariants are very often believed to be $0$, that is, $\cL_E^\Gr$ and $\cL_E^\MSD$ are not divisible by any $p$-power in $\Lambda$.

The vanishing of the (algebraic) $\mu$-invariant has a very specific arithmetic interpretation. If we denote by $\bQ_n$ the $n$-th layer in the tower $\bQ=\bQ_0\subset ...\subset \bQ_n\subset ...\subset \bQ_\infty$ where each $\bQ_n$ is such that $\#\Gal(\bQ_n/\bQ)=p^n$, then\[
\mu(\cL_E^\Gr)=0 \Leftrightarrow \Sha(E/\bQ_n)[p]\text{ is bounded as }n\to\infty.
\]
Here $\Sha(E/\bQ_n)$ is the Tate-Shafarevich group attached to $E$ over $\bQ_n$, one of the most mysterious objects in arithmetic geometry and Iwasawa theory.

Greenberg~\cite[Conjecture 1.11]{Gr99} has made very precise conjectures about the vanishing of the algebraic $\mu$-invariants, and Stevens~\cite[section 4]{Stevens} has formulated a compatible conjecture on the behavior of the analytic $\mu$-invariants. We will refer to the following statement as Greenberg's conjecture.

\begin{conjecture*}[Greenberg's conjecture]
    Let $E/\bQ$ be an elliptic curve. Assume that $\fX_\Gr(E/\bQ_\infty)$ is $\Lambda$-torsion. Then there is an elliptic curve $E'$ in the $\bQ$-isogeny class of $E$ such that $\mu(\cL_{E'}^\Gr)=0$. Whenever $\cL_E^\MSD$ (hence $\cL_{E'}^\MSD$) is well-defined, $\mu(\cL_{E'}^\MSD)=0$ for the same $E'$.
\end{conjecture*}

When $p$ is an ordinary prime for $E$, it is known that $\fX_\Gr(E/\bQ_\infty)$ is $\Lambda$-torsion and $\cL_E^\MSD$ exists. Thus Greenberg's conjecture predicts that we always have $\mu(\cL_{E'}^\Gr)=\mu(\cL_{E'}^\MSD)=0$ for some $E'$ which is $\bQ$-isogenous to $E$. In particular, if $E[p]$ is irreducible over $G_\bQ$ (so that the isogeny class consists of $E$ solely), it is predicted that $\mu(\cL_E^\Gr)=\mu(\cL_E^\MSD)=0$.

The main theorem of this paper is the following.
\begin{maintheorem}\label{mainthm}
    Let $E/\bQ$ be an elliptic curve, and let $p$ be an odd good ordinary Eisenstein prime for $E$. Then Greenberg's conjecture holds. When $E$ has multiplicative reduction at $p$, the algebraic part of Greenberg's conjecture holds.
\end{maintheorem}

When the cyclotomic Iwasawa Main Conjecture in~\cite{KY24} is extended to multiplicative primes, the analytic part of Greenberg's conjecture will automatically hold as well.

 When $p=2$ is a good ordinary Eisenstein prime for $E$, one also knows the algebraic part of Greenberg's conjecture (that is, the original conjecture~\cite[Conjecture 1.11]{Gr99}) holds. This is the main result of another paper by the authors~\cite{LY26a}.

 A key ingredient towards the proof of the main theorem for odd $p$ is the vanishing of the fine $\mu$-invariants of the $p$-th division field $\bQ(E[p])$ attached to $E$. More precisely, we have the following.

 \begin{maintheorem}\label{classical}
     Let $E/\bQ$ be an elliptic curve, and let $p$ be an odd ordinary Eisenstein prime for $E$. Then over the $p$-th division field $\bQ(E[p])$ of $E$, the following $\mu$-invariants vanish (see~\cref{introiwa} and~\cref{deffine} for definitions):
     \begin{itemize}
         \item[(i)] The classical $\mu$-invariant;
         \item[(ii)] The classical fine $\mu$-invariant;
         \item[(iii)] The fine $\mu$-invariant for $E$.
     \end{itemize}
 \end{maintheorem}

In fact, the equivalence of the (i) (which implies (ii)) and (iii) is already known in~\cite{CS05}, and at least point (i) should already be essentially known prior to Iwasawa~\cite{Iwa81}. Indeed, it is not hard to show that if the classical $\mu$ is $0$ for a number field $K$, then it is $0$ for a $p$-extension $L/K$. What's new is the observation that in the setting of the above theorem, $\bQ(E[p])$ is always a $p$-extension of an abelian field, where Ferrero--Washington theorem can be applied, and more importantly, that (i) often implies the vanishing of Greenberg's $\mu$ for $E$ over $\bQ$. It should be mentioned that we do not need (iii) to prove~\cref{mainthm}. For the purpose of self-containedness, we will include a proof of~\cref{classical} in~\cref{vanishingQ(E[p])}.

Our proof actually yields a more general result that can also be applied in the residually irreducible setting. Namely, the vanishing of Greenberg's $\mu$ over $\bQ$ can be deduced from the vanishing of some classical fine $\mu$ over certain extensions. More precisely, we have the following theorem, whose proof is given in~\cref{sec:4}.

\begin{maintheorem}\label{irred}
 Let $E$ be an elliptic curve and $p$ be an odd prime of ordinary reduction for $E$. Assume $E[p]$ is irreducible as a $G_\bQ$-module. Let $L=\bQ(E[p])$. Then $\mu(A^\fn_{L,\infty})=0\Rightarrow\mu(\Sel_\Gr(E/\bQ_\infty)^\vee)=0.$
 \end{maintheorem}

\subsection{Method of proof and outline of the paper}
When $p$ is an odd Eisenstein prime for $E$, there is a short exact sequence of $G_\bQ$-modules\begin{equation*}
    0\to \F(\phi) \to E[p] \to \F(\psi) \to 0,
\end{equation*}
where $\phi,\psi$ are characters $G_\bQ\to \GL_1(\F)$, and $\F(\theta)$ denotes the $G_\bQ$-module $\F$ with the action given by $\theta\in\{\phi,\psi\}$. One knows that $\phi\psi=\omega$, the mod $p$ cyclotomic character. Thus when $E$ is ordinary at $p$, exactly one of $\phi,\psi$ is unramified, and exactly one of them is odd.

When one of $\phi,\psi$ is unramified and odd (so the other is ramified and even), Greenberg and Vatsal~\cite{GV00} have proved that the algebraic and analytic $\mu$-invariants of $E$ are $0$, by passing to the study of the $\mu$-invariants of $\phi$ and $\psi$, where Ferrero--Washington theorem can be applied. When $\phi$ or $\psi$ is ramified and odd (say, if it equals $\omega$), however, Ferrero--Washington theorem does not apply and in fact one knows its (algebraic) $\mu$-invariant will be positive. It is not hard to see if $E[p]$ has a sub-representation which has positive $\mu$-invariant, $E$ will have positive $\mu$-invariant too. In fact, the algebraic part of Greenberg's conjecture can be reformulated as the following.

\begin{conjecture*}[algebraic part of Greenberg's conjecture for good ordinary primes, equivalent form]\label{equivform}
    Let $E/\bQ$ be an elliptic curve. Assume that $p$ is an ordinary prime for $E$. Then $\mu(\fX_\Gr(E/\bQ_\infty))=0$ in the following situations:
    \begin{itemize}
        \item $E[p]$ is an irreducible $G_\bQ$-module, or 
        \item $E[p]$ fits into a short exact sequence of $G_\bQ$-modules \begin{equation}\label{seq:E[p]}0\to \F(\phi) \to E[p] \to \F(\psi) \to 0,\end{equation}
        where\begin{enumerate}
            \item $\phi$ is unramified at $p$ and odd, or ramified at $p$ and even, or
            \item $\phi$ is unramified at $p$ and even (so $\psi$ is ramified at $p$ and odd), and the extension~\eqref{seq:E[p]} is non-split.
        \end{enumerate}
    \end{itemize}
\end{conjecture*}

That this formulation is equivalent to the original one~\cite[Conjecture 1.11]{Gr99} is a consequence of Schneider's formula~\cite[Corollary 1]{MR2152940}. See~\cref{Grconj} for more details. There is also an analytic version of Schneider's formula~\cite[Proposition 4.12]{Stevens}, but we will mainly deal with the algebraic version and only obtain results on the analytic side via Iwasawa Main Conjecture at the end of the paper.

The irreducible case seems out of reach as of now, and as is mentioned above, the subcase (i) was proved in~\cite{GV00} (although they only explicitly treated the good ordinary case. In terms of the algebraic $\mu$-invariant, the multiplicative case makes no difference). The rest of this subsection is devoted to an exposition of our proof of subcase (ii). For simplicity, we assume $\phi=\mathbf{1}$ and $\psi=\omega$.

Let $\Sigma$ be a finite set of places of $\bQ$ containing $p$, the Archimedean place $\infty$, and those where $E$ has bad reduction. Denote by $\bQ^\Sigma$ the maximal extension of $\bQ$ unramified outside $\Sigma$. Let $\Fil^+(E[p^\infty])\defeq\ker\big\{E[p^\infty]\to\tilde{E}[p^\infty]\big\}$ where $\tilde{E}$ is the reduction of $E$ at $p$. Then Greenberg's Selmer group $\Sel_\Gr(E/\bQ_\infty)$ is defined as the kernel of the following global-to-local map
\[\H^1(\bQ^\Sigma/\bQ_\infty, E[p^\infty])\to\prod_{w\in\Sigma,w\nmid p}\H^1(\bQ_{\infty,w},E[p^\infty])\times\frac{\H^1(\bQ_{\infty,w},E[p^\infty])}{\H^1_\Gr(\bQ_{\infty,p},E[p^\infty])},
\]
where $\H^1_\Gr(\bQ_{\infty,p},E[p^\infty])\defeq\ker\big\{H^1(\bQ_{\infty,p},E[p^\infty])\to\H^1(I_{\infty,p},E[p^\infty]/\Fil^+(E[p^\infty]))\big\},$ for $\bQ_{\infty,p}$ the completion of $\bQ$ at the unique place above $p$ which we also denote by $p$, and $I_{\infty,p}\subset G_{\bQ_{\infty,p}}$ the inertia subgroup at $p$.

It was observed in~\cite{GV00} the $\mu$-invariants of $\H^1(\bQ_{\infty,w},E[p^\infty])$ for $w\nmid p$ are $0$, so one can consider the following \emph{imprimitive} Selmer group\[
\Sel_\Gr^\Sigma(E/\bQ_\infty)\defeq\big\{\H^1(\bQ^\Sigma/\bQ_\infty, E[p^\infty])\to\frac{\H^1(\bQ_{\infty,p},E[p^\infty])}{\H^1_\Gr(\bQ_{\infty,p},E[p^\infty])}\big\},
\]
whose Pontryagin dual will have the same $\mu$-invariant as the primitive one.

Then one can notice that under the surjectivity of the above two global-to-local map defining the Selmer groups (also established in \emph{loc. cit.}), the vanishing (but not the exact value in general) of the $\mu$-invariant of $(\Sel^\Sigma_\Gr(E/\bQ_\infty))^\vee$ is equivalent to that of $(\Sel^\Sigma_\Gr(E[p]/\bQ_\infty))^\vee$, where $\Sel^\Sigma_\Gr(E[p]/\bQ_\infty)$ is the \emph{residual} imprimitive Selmer group defined as\[
\Sel^\Sigma_\Gr(E[p]/\bQ_\infty)\defeq\ker\big\{\H^1(\bQ^\Sigma/\bQ_\infty, E[p])\to\frac{\H^1(\bQ_{\infty,p},E[p])}{\H^1_\Gr(\bQ_{\infty,p},E[p])}\big\},
\]
where \[\H^1_\Gr(\bQ_{\infty,p},E[p])\defeq\ker\big\{H^1(\bQ_{\infty,p},E[p])\to\H^1(I_{\infty,p},E[p]/\Fil^+(E[p]))\big\},\]
with $\Fil^+(E[p])\defeq\Fil^+(E[p^\infty])[p].$

Recall that the ordinariness of $E$ implies that $E[p]$ has an unramified quotient (which is nothing but $\Fil^-(E[p])\defeq E[p]/\Fil^+(E[p])$) as a $G_{\bQ_p}$-module. Our assumption that the (non-direct) sub-$G_\bQ$-module of $E[p]$ is unramified at $p$ then forces $E[p]$ to be split over $G_{\bQ_p}$. That is, $E[p]=\F(\mathbf{1})\oplus \F(\omega)$ as $G_{\bQ_p}$-modules. This induces the splitting of local cohomology groups\[
\H^1(\bQ_{\infty,p},E[p])=\H^1(\bQ_{\infty,p},\F(\mathbf{1}))\oplus\H^1(\bQ_{\infty,p},\F(\omega)).
\]
We now explain one way to think about the vanishing of $\mu$-invariant of certain Selmer groups via tracking global-to-local images, which inspired our proof. It is solely for demonstration purpose, and we do not need the following vanishing of fine $\mu$-invariant as an input. Now the vanishing of \emph{fine} $\mu$-invariant of $E$ in~\cite{CS05} shows that the kernel of the global-to-local map\[
\H^1(\bQ^\Sigma/\bQ_\infty,E[p]) \xrightarrow{\loc}\H^1(\bQ_{\infty,p},E[p])=\H^1(\bQ_{\infty,p},\F(\mathbf{1}))\oplus\H^1(\bQ_{\infty,p},\F(\omega)) 
\]
has vanishing $\mu$-invariant, thus the only $\mu$-invariant of $\H^1(\bQ^\Sigma/\bQ,E[p])$ (which equals $1$ by~\cite[Proposition 5.8]{Gr99}) is completely mapped to $\H^1(\bQ_{\infty,p},E[p])$. On the other hand, the study of Greenberg's $\mu$-invariant is a finer task. Note that $\Fil^+(E[p])=\F(\omega)$ and $\Fil^-(E[p])=\F(\mathbf{1})$, so $\Sel^\Sigma_\Gr(E[p]/\bQ_\infty)$ can alternatively be written as \[\ker\big\{\H^1(\bQ^\Sigma/\bQ_\infty, E[p])\to\frac{\H^1(\bQ_{\infty,p},E[p])}{\H^1(\bQ_{\infty,p},\F(\omega))}\big\}.\]

The question is then to determine where the `$\mu=1$'-part of the global-to-local image lands in $\H^1(\bQ_{\infty,p},E[p])$. For example, if the `$\mu=1$'-part completely lands in $\H^1(\bQ_{\infty,p},\F(\omega))$ (meaning $\mu(\im(\loc)\cap\H^1(\bQ_{\infty,p},\F(\omega)))=1$), then this kernel is forced to have $\mu$-invariant equal to $1$. We need to show this is not the case in our situation.

To control the global-to-local image, one can go to the $p$-division field $\bQ(E[p])$ of $E$, where $E[p]$ splits into a direct sum $\F(\mathbf{1})\oplus\F(\omega)$. The key is to analyze the following commutative diagram\begin{center}{\small\begin{tikzcd}
			 \H^1(\bQ^\Sigma/\bQ_\infty,E[p]) \ar[r,"\res_{\bQ(E[p])/\bQ}"]\ar[d,"loc_\bQ"]&[3em] \H^1(\bQ(E[p])^\Sigma/\bQ(E[p])_\infty,E[p])\ar[d,"loc_{\bQ(E[p])}"]\\
			 \H^1(\bQ_{\infty,p},E[p])\ar[r,"\prod_{w_i\mid p}\res"]&[3em]\prod_{w_i\mid p}\H^1(\bQ(E[p])_{\infty,w_i},E[p]),\\	 	\end{tikzcd}}
\end{center}
which also reads as
\begin{center}{\small\begin{tikzcd}\label{gotoQ(E[p])}
			
			 \H^1(\bQ^\Sigma/\bQ_\infty,E[p]) \ar[r,"\res"]\ar[d,"loc_\bQ"]& \H^1(\bQ(E[p])^\Sigma/\bQ(E[p])_\infty,\F(\mathbf{1}))\oplus\H^1(\bQ(E[p])^\Sigma/\bQ(E[p])_\infty,\F(\omega)))\ar[d,"loc_{\bQ(E[p])}"]\\
			 \H^1(\bQ_{\infty,p},E[p])\ar[r,"\prod_{w_i\mid p}\res"]&\prod_{w_i\mid p}\H^1(\bQ(E[p])_{\infty,w_i},\F(\mathbf{1}))\oplus\H^1(\bQ(E[p])_{\infty,w_i},\F(\omega)).\\	 	\end{tikzcd}}
\end{center}

We will show that under the assumption that the sequence~\eqref{seq:E[p]} is non-split, \begin{itemize}
    \item[(i)] the `$\mu=1$'-part of $\im(\res_{\bQ(E[p])/\bQ})$ doesn't fully land in $\H^1(\bQ(E[p])_\infty,\F(\omega))$;
    \item[(ii)] Further, the `$\mu=1$'-part of $\loc_{\bQ(E[p])}(\im(\res_{\bQ(E[p])/\bQ}))$ doesn't fully land in $\prod_{w_i\mid p}\H^1(\bQ(E[p])_{\infty,w_i},\F(\omega))$.
\end{itemize}  Then from the commutativity of this diagram, we know the `$\mu$=1'-part of $\im(\loc_\bQ)$ cannot fully land in $\H^1(\bQ_{\infty,p},\F(\omega))$, which is exactly what we desire. See~\cref{sec:3} for more precise statements.

To make sense of (i), one uses the fact that $\res_{\bQ(E[p])/\bQ}$ has finite kernel, and the non-splitness of~\eqref{seq:E[p]}. To make sense of (ii), one needs the vanishing of classical fine $\mu$-invariant of $\bQ(E[p])$ , which in our case is equivalent to the vanishing of $\mu$-invariant of \[\ker\big\{\H^1(\bQ(E[p])_\infty,\F(\mathbf{1}))\to\prod_{w_i\mid p}\H^1(\bQ(E[p])_{\infty,w_i},\F(\mathbf{1}))\big\}.\]

Finally, we use the Iwasawa Main Conjecture to confirm the behavior of analytic $\mu$-invariants predicted by Stevens.

In the first $2$ sections, we discuss the basics about Selmer groups, $p$-adic $L$-functions and in general, Iwasawa theory. In~\cref{sec:3}, we provide the ingredients into the above proof, including the vanishing of classical fine $\mu$-invariants of $\bQ(E[p])$ needed in step (ii) above. Finally, in~\cref{sec:4}, we discuss some potential applications of our arguments in the irreducible setting.

\subsection{Relation to previous works}
It seems reasonable to believe that Greenberg's conjecture for Selmer groups for elliptic curves was posed in its first explicit form during the period between the two works of Greenberg~\cite[formula (75)]{Gr89} and~\cite[Conjecture 1.11]{Gr99}. At the same time, Stevens~\cite[Remark (4.14)]{Stevens} took a compatible guess on the behavior of $\mu$-invariants for $p$-adic $L$-functions. Ever since then, little progress has been made in the Eisenstein case, except in Trifkovi\'c's work~\cite{MR2152940} where he utilized methods from flat cohomology in successfully finding infinitely many non-isomorphic elliptic curves satisfying Greenberg's conjecture.

In a different direction, Hachimori~\cite[Theorem 6.4]{Hac04} provided a link between the vanishing of $\mu$-invariants of certain elliptic curves with reducible residual representation and the vanishing of $\mu$-invariants of certain number field. 

Our method is different from Trifkovi\'c's and Hachimori's. We work with usual Galois cohomology and study the vanishing of $\mu$-invariant of the elliptic curve directly, using the vanishing results for relevant classical $\mu$-invariants and fine $\mu$-invariants, where~\cite{CS05} is an input. Moreover, the aforementioned papers only dealt with the special case where $E[p]$ is a non-split extension of $\F(\omega)$ by $\F(\mathbf{1})$ for small primes $p=3$ (\cite{Hac04},~\cite{MR2152940}) and $p=5$ (\cite{MR2152940}), and assumed good ordinary reduction, whereas our results apply to any odd $p$ and arbitrary semisimplification of $E[p]$, allowing multiplicative reduction as well.

On the other hand, Ray~\cite{ray23} did some valid attempts in relating Greenberg's Selmer groups to fine Selmer groups, under some conjectural assumptions. Our result seems to confirm his conjecture as well.

\subsection{Further developments}
As Greenberg originally posed his conjecture for ordinary $p$-adic representations, it is natural to ask if the conjecture is true more generally, such as when the representation comes from higher weight modular forms. Such extensions should not be difficult and will be addressed in later work.

The key philosophy in our work is that $E[p]$, or in fact $\bQ(E[p])$, is enough to determine the vanishing of Greenberg's $\mu$-invariants, which has its origin in the work of Greenberg and Vatsal. In the last section of the paper we make it precise, and leave room for further developments in the study of such fields for more general elliptic curves.

In this work, we eventually utilized the Iwasawa Main Conjecture to obtain information on the $\mu$-invariants of $p$-adic $L$-functions. It would be interesting to see if one can study these analytic $\mu$-invariants directly.

It would also be natural to ask if one can obtain computational results like those in~\cite{LY26b}, where one can directly read from the arithmetics of $E$ information about its $\mu$-invariants, or obtain an absolute upper bound of the $\mu$-invariant. Indeed, like in the $p=2$ case, our work confirms Greenberg's conjecture that the $\mu$-invariants arise solely from the existence of ramified and odd $G_\bQ$-invariant submodule of $E[p^\infty]$, so there is a natural upper bound $\mu_E\leq n$ if $E[p^n]$ is reducible while $E[p^{n+1}]$ becomes irreducible for $n\geq 1$. An absolute upper bound would follow from an analysis of possible degrees of general $p$-isogenies.

\subsection{Notations and conventions}
For any field $L$, we denote by $G_L$ its absolute Galois group, and we write $\H^i(L,-)$ for $\H^i(G_L,-)$. If $L/K$ is an extension of fields, we also write $\H^i(L/K,-)$ for $\H^i(\Gal(L/K),-)$. Denote by $(-)^\vee=\Hom_\cts(-,\bQ_p/\bZ_p)$ the Pontryagin dual. 
\subsection{Acknowledgments}
The authors thank Francesc Castella for pointing out the works by Trifkovi\'c and Stevens and other helpful comments, and Tom Weston for helpful discussions and encouragements.

\addtocontents{toc}{\protect\setcounter{tocdepth}{2}}
\section{Preliminaries}
\subsection{Division field for elliptic curves}\label{Q(E[p])}
Let $E/\bQ$ be an elliptic curve, and $E[p]$ be its $p$-torsion for an odd $p$. In this subsection we will study useful properties of the $p$-th division field $\bQ(E[p])$ of $E$, the field by adjoining the coefficients of the $p$-torsion points of $E$. In particular, $G_{\bQ(E[p])}=\ker(G_\bQ\to \Aut(E[p]))$, and $\bQ(E[p])$ is Galois over $\bQ$ with $\Gal(\bQ(E[p])/\bQ)$ isomorphic to a (generally non-abelian) subgroup of $\GL_2(\F)$. 

When $p$ is an Eisenstein prime for $E$, $E[p]$ has the form
$\begin{pmatrix}
\phi & * \\
0 & \psi 
\end{pmatrix}$   
where $\phi,\psi: G_\bQ\to \GL_1(\F)$ are characters with $\phi\psi=\omega$. It is helpful to study the intermediate field $\bQ(\phi,\psi)$ defined as the composite of the fixed fields of $\ker(\phi)$ and $\ker(\psi)$. $\bQ(\phi,\psi)$ is then an abelian extension of $\bQ$ and $E[p]$ becomes $\begin{pmatrix}
\mathbf{1} & * \\
0 & \mathbf{1} 
\end{pmatrix}$
over $\bQ(\phi,\psi)$. Thus the index $[\bQ(E[p]):\bQ(\phi,\psi)]$ is either $1$ or $p$ depending on whether * is $0$ or not. As $\phi,\psi$ have order dividing $p-1$, the degree $[\bQ(\phi,\psi):\bQ]$ divides $(p-1)^2$ and is coprime to $p$. Moreover, when $\bQ(E[p])\ne\bQ(\phi,\psi)$, $\bQ(E[p])$ is not abelian over $\bQ$.

\subsection{Classical Iwasawa theory}\label{introiwa}
Let $\bQ_\infty=\bQ(\mu_{p^\infty})^\Delta$ be the cyclotomic $\bZ_p$-extension of $\bQ$, where $\Delta=\Gal(\bQ(\mu_p)/\bQ)$, and let $L_\infty\defeq L\bQ_\infty$ be the cyclotomic $\bZ_p$-extension of any number field $L$. Let $\Gamma\defeq \Gal(\bQ_\infty/\bQ)$, and denote by $\Lambda\defeq \bZ_p\llbracket \Gamma\rrbracket$ the Iwasawa algebra. We will always identify $\Gamma$ with a formal power series ring $\bZ_p\llbracket T\rrbracket$, where a topological generator $\gamma\in\Gamma$ is identified with $1+T$. The Iwasawa algebra is not a principal ideal domain, yet there is a nice structure theorem for finitely generated $\Lambda$-modules. If $X$ is any such module, then there is a $\Lambda$-linear morphism with finite kernel and cokernel (known as a \emph{pseudo-isomorphism}) from $X$ to a module of the form\[
\Lambda^r\oplus\bigoplus_{i}\Lambda/(f_i)\oplus\bigoplus_j\Lambda/(p^{m_j})
\]
for some $r$, $i$ and $j$, where $f_i$'s are \emph{distinguished polynomials} satisfying $f_i\equiv T^{\deg(f_i)}\pmod{p}$. Here the values $r$, $\Sigma_i \deg(f_i)$, $\Sigma_j m_j$ are the \emph{rank}, \emph{$\lambda$-invariant} and the \emph{$\mu$-invariant} of $X$, although some authors might (and in practice we do) require $r=0$ to define $\lambda$ and $\mu$-invariants.

Typically, one can attach a $\Lambda$-module to various arithmetic objects, such as number fields and elliptic curves. These special $\Lambda$-modules are usually known as~\emph{Selmer groups}. One major goal in Iwasawa theory is to study the structures of Selmer groups. We begin with some examples that can be found in~\cite[Introduction]{Sha}.

\begin{example}
    Let $K$ be a number field, and let $K_\infty\defeq K\bQ_\infty$ be its cyclotomic $\bZ_p$ extension. Then there is a tower $K\defeq K_0\subset K_1\subset ...\subset K_n\subset...$ with $[K_n:K]=p^n$ and $K_\infty=\cup_n K_n$. 
    Denote by $\Cl(K_n)[p^\infty]$ the \emph{$p$-primary} part of the class group of the $n$-the layer. Then $\varprojlim_n \Cl(K_n)[p^\infty]$ is a finitely generated torsion module over $\Lambda_K\defeq\bZ_p\llbracket\Gal(K_\infty/K) \rrbracket\iso \bZ_p\llbracket T\rrbracket$, and one can talk about its $\mu$-invariant. 
\end{example}

\begin{example}
    Let $L_\infty/K_\infty$ the maximal unramified abelian pro-$p$ extension. Denote by $A_{K,\infty}=\Gal(L_\infty/K_\infty)$. There is a canonical isomorphism $\varprojlim_n \Cl(K_n)[p^\infty]\iso A_{K,\infty}$.
\end{example}

\begin{definition}
    Let $K$ be an algebraic extension of $\bQ$, $v$ a finite place of $K$. Let $K_v^\ur$ be the maximal unramified extension of $K_v$. Then $\Gal(K_v^\ur/K_v)\iso G_{K_v}/I_v$ where $I_v\isom \Gal(\ol K_v/K_v^\ur)$ is the inertia subgroup.
\end{definition}

\begin{lemma}
    Assume $K$ is a finite extension over $\bQ$ and let $\Sigma$ be any finite set of primes containing $p$. Then \begin{align*}       
    A_{K,\infty}&\isom
    \ker\big\{\H^1(K^\Sigma/K_\infty,\bQ_p/\bZ_p)\to\prod_v \H^1(K_{\infty,v}, \bQ_p/\bZ_p)/\H^1(K_{\infty,v}^\ur/K_{\infty,v}, \bQ_p/\bZ_p)\big\}^\vee\\
    &\isom \ker\big\{\H^1(K^\Sigma/K_\infty,\bQ_p/\bZ_p)\to\prod_v \H^1(I_{\infty,v}, \bQ_p/\bZ_p)\big\}^\vee
    \end{align*}
\end{lemma}

\begin{proof}
    The first isomorphism is proposition 1.6.1 of \cite{Ru00}. The second isomorphism follows from the inflation-restriction exact sequence. It is independent from the choice of $\Sigma$ by the remark below.
\end{proof}

\begin{remark}\label[remark]{unrawayp}
    $A_{K,\infty}$ can be identified with\begin{align*}
    \ker\big\{\H^1(K^\Sigma/K_\infty,\bQ_p/\bZ_p)\to&\prod_{v\mid p}\H^1(K_{\infty,v}, \bQ_p/\bZ_p)/\H^1(K_{\infty,v}^\ur/K_{\infty,v}, \bQ_p/\bZ_p)\\ 
    \cdot&\prod_{v\nmid p} \H^1(K_{\infty,v}, \bQ_p/\bZ_p) \big\}^\vee,
    \end{align*} since for $v\nmid p$, $\Gal(K^\ur_{\infty,v}/K_{\infty,v})$ has profinie degree prime to $p$. In other words, unramified local condition at places away from $p$ could be replaced with \emph{strict} local condition.
\end{remark}

$A_{K,\infty}$ has a quotient module \[A^\fn_{K,\infty}=
    \ker\big\{\H^1(K^\Sigma/K_\infty,\bQ_p/\bZ_p)\to\prod_v\H^1(K_{\infty,v}, \bQ_p/\bZ_p)
  \big\}^\vee,
    \]which we call the \emph{fine} Selmer group for the field $K$, whose $\mu$-invariant we call `\emph{classical fine $\mu$ for $K$}'.
    For later use, we also define the $\Sigma$-imprimitive Selmer group and its $p$-torsion version
  \[A^{\fn,\Sigma}_{K,\infty}=
    \ker\big\{\H^1(K^\Sigma/K_\infty,\bQ_p/\bZ_p)\to\prod_{v\mid p}\H^1(K_{\infty,v}, \bQ_p/\bZ_p)
  \big\}^\vee,
    \]

    \[A[p]^{\fn,\Sigma}_{K,\infty}=\ker\big\{\H^1(K^\Sigma/K_\infty,\F(\mathbf{1}))\to\prod_{v\mid p}\H^1(K_{\infty,v}, \F(\mathbf{1}))
  \big\}^\vee,\]
  where $\F(\mathbf{1})$ denotes the $G_K$-module $\F$ with trivial action.

\begin{theorem}\label{Ferrero-Washington}
    Assume $K/\bQ$ is an abelian extension. Then $A_{K,\infty}$ is a finitely generated $\Lambda_K$-torsion module with $\mu$-invariant 0. Consequently, $A^\fn_{K,\infty}$ is a finitely generated $\Lambda_K$-torsion module with $\mu$-invariant 0.
\end{theorem}
\begin{proof}
    The first claim is the main result of \cite{FW79}. The consequence is immediate.
\end{proof}

\section{Iwasawa theory for $E$}

\subsection{Selmer groups for $E$}
In this subsection, we always assume that $E$ has ordinary reduction at $p$ in the sense of Greenberg~\cite{Gr89}. We omit the subscript for $\Lambda$ when it is clear from the context.

\subsubsection{Greenberg's Selmer group for $E$}
\begin{definition}
    Let $\wt{E}$ be the reduction curve of $E$ mod $p$, and let $I_{\infty,p}$ be the inertia group of $\bQ_{\infty,p}$. Define \[\Fil^+(E[p^\infty])=\ker(E[p^\infty]\to \wt{E}[p^\infty]),\] $\Fil^-(E[p^\infty])=E[p^\infty]/\Fil^+(E[p^\infty])$, $\Fil^+(E[p])=\Fil^+(E[p^\infty])[p]$ and $ \Fil^-(E[p])=E[p]/\Fil^+(E[p])$.
\end{definition}
\begin{definition}
    Let $\Sigma$ be the finite set of primes containing $p$ and the primes where $E$ has bad reduction. Define \[ \Sel_\Gr(E/\bQ_\infty)=\ker(\H^1(\bQ^\Sigma/\bQ_\infty, E[p^\infty])\to \prod_{v\in \Sigma-p} \H^1(\bQ_{\infty,v}, E[p^\infty])\times \H^1(I_{\infty,p}, \Fil^-(E[p^\infty])),\] \[ \Sel^\Sigma_\Gr(E/\bQ_\infty)=\ker(\H^1(\bQ^\Sigma/\bQ_\infty, E[p^\infty])\to  \H^1(I_{\infty,p}, \Fil^-(E[p^\infty])),\]
    \[\Sel_\Gr(E[p]/\bQ_\infty)=\ker(\H^1(\bQ^\Sigma/\bQ_\infty, E[p])\to \prod_{v\in \Sigma-p} \H^1(\bQ_{\infty,v}, E[p])\times \H^1(I_{\infty,p}, \Fil^-(E[p])),\]
    \[\Sel^\Sigma_\Gr(E[p]/\bQ_\infty)=\ker(\H^1(\bQ^\Sigma/\bQ_\infty, E[p])\to  \H^1(I_{\infty,p}, E[p])).\]
\end{definition}

Our definitions might seem different from the standard ones at places away from $p$ by an unramified local condition (or \emph{finite} local condition for $E[p]$ coming from propagation of unramified local condition), but from~\cref{unrawayp}, the Selmer groups agree.

The following facts about Greenberg's Selmer groups are well-known.

\begin{proposition}\label[proposition]{Greenberg Selmer torsion}
    Assume $E$ has ordinary reduction at $p$. Then $\Sel_\Gr(E/\bQ_\infty)$ and $\Sel_\Gr^\Sigma(E/\bQ_\infty)$ are finitely generated $\Lambda$-cotorsion modules
\end{proposition}

\begin{lemma}\label[lemma]{Sigma part mu=0}
    $\prod_{v\in \Sigma-p} H^1(\bQ_{\infty,v}, E[p^\infty])^\vee$ is a $\Lambda$-torsion module with $\mu$-invariant $0$. More generally, if $K$ is a number field, $\prod_{w\mid v,v\in \Sigma-p} H^1(K_{\infty,w}, E[p^\infty])^\vee$ is a $\Lambda$-torsion module with $\mu$-invariant $0$.
\end{lemma}

\begin{proof}
    When $K=\bQ$, this is Proposition 2.4 of \cite{GV00}. The general case follows similarly from~\cite[Proposition 2]{Gr89}, since $\H^0(K_{\infty,w},E[p^\infty])$ is $\bZ_p$-cofinitely generated for any finite extension $K/\bQ$ and $w\nmid p$. Indeed, for primes $w$ of good reduction, $\H^0(K_{\infty,w},E[p^\infty])$ is finite, and for bad primes, it is at most of $\bZ_p$-corank $2$. Note that we have implicitly used the fact that no place can decompose infinitely along cyclotomic $\bZ_p$-extensions.
\end{proof}

\begin{proof}(\cref{Greenberg Selmer torsion})
    The first part follows from theorem 17.4 of \cite{Ka04}. Now consider the following exact sequence
    \begin{equation*}
        0\to \Sel_\Gr(E/\bQ_\infty)\to \Sel_\Gr^\Sigma(E/\bQ_\infty)\to \prod_{v\in \Sigma-p}H^1(\bQ_{\infty,v},E[p^\infty]).
    \end{equation*} The second part follows from \cref{Sigma part mu=0}.
\end{proof}

\subsubsection{The fine Selmer group for $E$}\label{deffine}
\begin{definition}
    Let $\Sigma$ be a finite set of primes containing $p$ and the primes where $E$ has bad reduction. Define \[ \Sel_\fn(E/\bQ_\infty)=\ker(\H^1(\bQ^\Sigma/\bQ_\infty, E[p^\infty])\to \prod_{v\in \Sigma} \H^1(\bQ_{\infty,v}, E[p^\infty])),\] \[ \Sel^\Sigma_\fn(E/\bQ_\infty)=\ker(\H^1(\bQ^\Sigma/\bQ_\infty, E[p^\infty])\to  \H^1(\bQ_{\infty,p}, E[p^\infty])),\]
    \[\Sel_\fn(E[p]/\bQ_\infty)=\ker(\H^1(\bQ^\Sigma/\bQ_\infty, E[p])\to \prod_{v\in \Sigma} \H^1(\bQ_{\infty,v}, E[p])),\]
    \[\Sel^\Sigma_\fn(E[p]/\bQ_\infty)=\ker(\H^1(\bQ^\Sigma/\bQ_\infty, E[p])\to  \H^1(\bQ_{\infty,p}, E[p])).\]
\end{definition}

For any number field $K$, one can define the fine Selmer groups of $E$ over $K_\infty\defeq K\bQ_\infty$ by replacing $\bQ_\infty$ by $K_\infty$ and replacing $\prod_{v\in \Sigma}\H^1(\bQ_{\infty,v},-)$ (\emph{resp. $\H^1(\bQ_{\infty,p},-)$}) by $\prod_{w\mid v,v\in \Sigma}\H^1(K_{\infty,w},-)$ (\emph{resp. $\prod_{w\mid p}\H^1(K_{\infty,w},-)$}). The following facts about the fine Selmer groups are well-known.

\begin{proposition}\label[proposition]{fine Selmer mu=0}
    $\Sel_\fn(E/\bQ_\infty)^\vee$ and $\Sel_\fn^\Sigma(E/\bQ_\infty)^\vee$ are finitely generated $\Lambda$-torsion modules with $\mu$-invariant 0.
\end{proposition}

\begin{theorem}\label{Coates-Sujatha}
    Let $p\geq 3$ a prime. Assume there exists a finite extension $F/\bQ$ such that \begin{enumerate}
        \item $F\subset \bQ(E[p^\infty])$.
        \item $\bQ(E[p^\infty])/F$ a pro-$p$ extension.
        \item $F_\infty:=F\bQ_\infty$ has classical Iwasawa $\mu$-invariant 0.
    \end{enumerate} Then $\Sel_\fn(E/\bQ_\infty)$ is a finitely generated $\Lambda$-cotorsion module with $\mu$-invariant 0. 
\end{theorem}

\begin{proof}
    This is Corollary 3.5 of \cite{CS05}.
\end{proof}

\begin{proof}(\cref{fine Selmer mu=0})
    Suppose we have a short exact sequence of $G_\bQ$-modules \[0\to \F(\varphi)\to E[p]\to \F(\psi)\to 0.\] Take $F=\bQ(\varphi,\psi)$ as in~\cref{Q(E[p])}. Then $F\subset \bQ(E[p])\subset \bQ(E[p^\infty])$ and $[\bQ(E[p]):F]=1$ or $p$. Since $\bQ(E[p^\infty])/\bQ(E[p])$ is a pro-$p$ extension, so is $\bQ(E[p^\infty])/F$. Since $F/\bQ$ is abelian. According to \cref{Ferrero-Washington}, $F_\infty$ has classical Iwasawa $\mu$-invariant 0. Therefore, by \cref{Coates-Sujatha}, $\Sel_\fn(E/\bQ_\infty)^\vee$ is a finitely generated $\Lambda$-torsion module with $\mu$-invariant 0. The second part follows from the short exact sequence  \begin{equation*}
        0\to \Sel_\fn(E/\bQ_\infty)\to \Sel_\fn^\Sigma(E/\bQ_\infty)\to \prod_{v\in \Sigma-p}H^1(\bQ_{\infty,v},E[p^\infty])
    \end{equation*} and \cref{Sigma part mu=0}.
\end{proof}

\begin{remark}\label[remark]{finemuK}
    A similar sequence for a general number field $K$, namely,\[
0\to \Sel_\fn(E/K_\infty)\to \Sel_\fn^\Sigma(E/K_\infty)\to \prod_{w\mid v,v\in \Sigma-p}\H^1(K_{\infty,w},E[p^\infty]),
    \]
   together with~\cref{Sigma part mu=0} shows that $\corank_\Lambda(\Sel_\fn(E/K_\infty))=\corank_\Lambda(\Sel^\Sigma_\fn(E/K_\infty))$ and $\mu(\Sel_\fn(E/K_\infty)^\vee)=\mu(\Sel^\Sigma_\fn(E/K_\infty)^\vee)$.
\end{remark}

\subsection{The $p$-adic $L$-function for $E$}
Let $E/\bQ$ be an elliptic curve and let $p$ be a prime of good ordinary reduction for $E$
. For a Dirichlet character $\chi$, denote by $L(E/\bQ,\chi,s)$ the Hasse--Weil $L$-series attached to $E$, twisted by $\chi$, and let $\Omega_E$ be the real N\'eron period for $E$. Let $\alpha_p$ be the unit root of the Hecke polynomial\[
x^2-a_p(E)x+p,
\]where $a_p(E)=p+1-\#\tilde{E}(\F)$.

Mazur and Swinnerton-Dyer~\cite{MSD74} have constructed an element $\cL_E^\MSD(T) \in\Lambda\otimes \bQ_p$, known as the (Mazur--Swinnerton-Dyer) $p$-adic $L$-function, satisfying the following interpolation property: for any finite order character $\chi$ of $\Gamma$ of conductor $p^r$ with $r >0$, we have\[
\cL_E^\MSD(\chi)=\frac{p^r}{\tau(\ol\chi) \alpha_p^r}\cdot \frac{L(E/\bQ,\chi,1)}{\Omega_E},
\]
where $\tau(\ol\chi)=\sum_{a\pmod{p^r}}\ol\chi(a)e^{2\pi ia/p^r}$ denotes the Gauss sum, and \[
\cL_E^\MSD(\mathbf{1})=(1-\alpha_p^{-1})\cdot \frac{L(E/\bQ,1)}{\Omega_E}.
\] 
\begin{lemma}[\cite{GV00},~\cite{Wut17}]
    $\cL_E^\MSD$ is integral, i.e., $\cL_E^\MSD\in\Lambda$.
\end{lemma}
\begin{proof}
    See for example~\cite[Theorem 2.1.1]{CGS}.
\end{proof}

\subsection{Iwasawa invariants and Greenberg's conjecture}\label{Grconj}
In~\cref{introiwa}, we have defined the Iwasawa $\mu$ and $\lambda$-invariants of finitely generated torsion $\Lambda$-module $X$, via the structure theorem. In particular, if $X$ is such a module, we denote by \[\Char(X)=p^{\mu(X)}\cdot\prod_i f_i\] its \emph{characteristic polynomial}, and define the \emph{characteristic ideal} of $X$ to be
$\Char_\Lambda(X)$, the ideal generated by $\Char(X)$.

Analogously, for any element $\cL$ of $\Lambda$ of finite degree, we can define its $\mu$-invariant to be \[
\mu(\cL)=\max\{n\in\bZ^{\geq 0}:p^n\mid\cL\},
\]and its $\lambda$-invariant to be\[
\lambda(\cL)=\deg(\cL).
\]
We often refer to $\mu^\alg(E)\defeq\mu(\Sel_\Gr(E/\bQ_\infty)^\vee)$ and $\lambda^\alg(E)\defeq\lambda(\Sel_\Gr(E/\bQ_\infty)^\vee)$ as the \emph{algebraic} $\mu$ and $\lambda$-invatiants of $E$, and refer to $\mu^\an(E)\defeq\mu(\cL_p^\MSD)$ and $\lambda^\an(E)\defeq\lambda(\cL_p^\MSD)$ as the \emph{analytic} $\mu$ and $\lambda$-invariants.

Recall the statement of Greenberg's conjecture from the introduction.

\begin{conjecture*}[Greenberg's conjecture]
    Let $E/\bQ$ be an elliptic curve. Assume that $\fX_\Gr(E/\bQ_\infty)$ is $\Lambda$-torsion. Then there is an elliptic curve $E'$ in the $\bQ$-isogeny class of $E$ such that $\mu(\cL_{E'}^\Gr)=0$. Whenever $\cL_E^\MSD$ (hence $\cL_{E'}^\MSD$) is well-defined, $\mu(\cL_{E'}^\MSD)=0$ for the same $E'$.
\end{conjecture*}
In other words, whenever $\Sel_\Gr(E/\bQ_\infty)$ is $\Lambda$-cotorsion, there should exist a curve $\bQ$-isogenous to $E$ with vanishing algebraic and analytic $\mu$-invariants. 

On the other hand, Schneider~\cite{Sch87} (followed by Perrion-Riou~\cite{PR89}, and Drinen~\cite{Dri02} to include $p=2$) studied how (algebraic) $\mu$-invariant changes under isogeny. We only need the following simple version.

\begin{lemma}\label[lemma]{Schformula}
    Let $\phi$ be a $\bQ$-isogeny from $E$ to $E'$ with $C\defeq\ker(\phi)$ fitting into the short exact sequence\[
    0\to C\to E\xrightarrow{\phi} E'\to 0,
    \]
    so $C$ is cyclic of order $p^n$ for some $n\in\bZ^{\geq 0}$. Then \[\mu^\alg(E')=\mu^\alg(E)+\delta,
    \]
    where \begin{equation*}
        \delta=\begin{cases}
            -n&C\text{ is ramified and odd;}\\
            n&C\text{ is unramified and even;}\\
            0&otherwise.
        \end{cases}
    \end{equation*}
    Here $C$ is ramified (\emph{resp. odd}) if and only if the action of $I_p$ (\emph{resp. complex conjugation}) on $C[p]$ is non-trivial.
\end{lemma}
\begin{proof}
    See~\cite[Lemma 1]{MR2152940}.
\end{proof}
From the above lemma, it is not hard to see that any curve with minimal $\mu$-invariant in an isogeny class is one that does not admit a $\bQ$-isogeny whose kernel is ramified and odd, and Greenberg's conjecture predicts exactly that $\mu=0$ for such curves. In other words, Greenberg's conjecture predicts that all (algebraic) $\mu$-invariants arise from $G_\bQ$-invariant submodules of $E[p^\infty]$ that are ramified and odd.

We therefore obtain the following equivalent form of the algebraic part of Greenberg's conjecture as in the introduction.\begin{conjecture*}
    Let $E/\bQ$ be an elliptic curve. Assume that $p$ is an ordinary prime for $E$. Then $\mu(\fX_\Gr(E/\bQ_\infty))=0$ in the following situations:
    \begin{itemize}
        \item $E[p]$ is an irreducible $G_\bQ$-module, or 
        \item $E[p]$ fits into a short exact sequence of $G_\bQ$-modules \begin{equation}\label{seq:E[p]in2}0\to \F(\phi) \to E[p] \to \F(\psi) \to 0,\end{equation}
        where\begin{enumerate}
            \item $\phi$ is unramified at $p$ and odd, or ramified at $p$ and even, or
            \item $\phi$ is unramified at $p$ and even (so $\psi$ is ramified at $p$ and odd), and the extension~\eqref{seq:E[p]in2} is non-split.
        \end{enumerate}
    \end{itemize}
\end{conjecture*}

As case (i) was completely solved in~\cite{GV00} for odd $p$ (see the remark below) and in~\cite{LY26a} for $p=2$ (which also treats case (ii)), the main goal of~\cref{sec:3} is to study case (ii) for odd primes. 

\begin{remark}
\begin{enumerate}
    \item In~\cite{GV00}, only the good ordinary case was explicitly written, although at the beginning of the text it was pointed out that the theorems apply to elliptic curves with multiplicative reductions as well. Indeed, at least on the algebraic side, the proof of $\mu_{E,\Sigma_0}^\alg=\mu_{\phi,\Sigma_0}^\alg+\mu_{\psi,\Sigma_0}^\alg=0$ (hence $\mu_{E}^\alg=0$) goes through almost verbatim.
    \item It was also noted that (see also~\cite[section 2]{Gr89}) in the split multiplicative case, the classical Selmer group defined using Kummer image is strictly smaller than the one defined above using ordinary local conditions. But the difference only has $\lambda$-invariant $1$, corresponding to a `trivial zero', and has no $\mu$-invariant. Therefore we will stick to the ordinary Selmer groups when we refer to the algebraic part of Greenberg's conjecture.
\end{enumerate}

\end{remark}

We finish this section with a discussion on the analytic version of Greenberg's conjecture, first posed by Stevens in~\cite[Corollary (4.13), Remark (4.14)]{Stevens}. Just as on the algebraic side, there is a Schneider type formula describing how the (analytic) $\mu$-invariant changes under isogeny. From Proposition (4.12) in \textit{loc. cit.} applied to $\cL_p^\MSD$ (namely, the $i=0$ case), there is an increase in $\mu$-invariant from the \emph{minimal curve} $E_\min$ if and only if the isogeny $E \to E_\min$ has an odd and ramified kernel (recall that the dual isogeny $E_{\min} \to E$ always has an \'etale kernel by the characterization of $E_\min$), and the difference in $\mu$ agrees with Greenberg's prediction on the algebraic side.

The relation between the algebraic $\mu$-invariant of $E$ and the analytic $\mu$-invariant of $E$ is provided by the Iwasawa Main Conjecture for $E$, which predicts exactly that $\mu^\alg(E)=\mu^\an(E)$ (and much more, see~\cref{IMC}). Under this assumption, the conjecture due to Stevens is compatible with Greenberg's conjecture, and it follows immediately that, in fact, Greenberg's conjecture predicts the minimal curve should always have (analytic) $\mu$-invariant $0$. Thus Steven's conjecture is somewhat stronger than Greenberg's conjecture as it is written, in that it specifies one candidate for the desired curve with vanishing $\mu$-invariant. The Iwasawa Main Conjecture itself is very deep, but luckily it is known in the cases we are concerned with, so we can use information on the algebraic side to get information on the analytic side. It would be interesting to see if Steven's conjecture could be solved without using Iwasawa Main Conjecture.

\section{Proof of Greenberg's Conjecture at Eisenstein primes}\label{sec:3}
In this section, $E$ will be an ellipit curve with ordinary reduction at an Eisenstein prime $p$. For a number field $L$, $\Sigma$ will denote a finite set of places containing those above $p$ and those where $E$ has bad reduction. 
\subsection{The vanishing of algebraic $\mu$-invariants}
\subsubsection{Greenberg--Vatsal type results}
\begin{proposition}\label[proposition]{finemumodp}    $\mu(\Sel_\fn(E/\bQ_\infty)^\vee)=0\Leftrightarrow \mu(\Sel^\Sigma_\fn(E[p]/\bQ_\infty)^\vee)=0.$
\end{proposition}
\begin{proof}
Consider the diagram{\footnotesize
	\begin{equation*}\begin{tikzcd}
			0 \ar[r]& \ker(\loc^\fn_{/p}) \ar[r]\ar[d]&\Sel^\Sigma_\fn(E[p]/\bQ_\infty)  \ar[r]\ar[d]&  \Sel^\Sigma_\fn(E/\bQ_\infty)[p] \ar[d] &\\
			0 \ar[r]&  \H^0(\bQ^\Sigma/\bQ_\infty,E[p^\infty])/p\ar[r]\ar[d,"\loc^\fn_{/p}"]& \H^1(\bQ^\Sigma/\bQ_\infty,E[p])\ar[r]\ar[d]& \H^1(\bQ^\Sigma/\bQ_\infty,E[p^\infty])[p]
            \ar[r]\ar[d]&0\\
			0 \ar[r]& \H^0(\bQ_{\infty,p}, E[p^\infty])/p \ar[r]\ar[d]& \H^1(\bQ_{\infty,p}, E[p]) \ar[r]& \H^1(\bQ_{\infty,p}, E[p^\infty])[p]\ar[r]& 0 \\
            &\coker(\loc^\fn_{/p})&&&
\end{tikzcd}\end{equation*}}
coming from the sequence\begin{equation}\label{E[p^infty]}
    0\to E[p] \to E[p^\infty] \xrightarrow{p} E[p^\infty]\to 0,\end{equation}where $\ker(\loc^\fn_{/p})$ and $\coker(\loc^\fn_{/p})$ are finite since the domain and codomain of $\loc^\fn_{/p}$ are finite. It then follows immediately from the snake lemma that $\Sel^\Sigma_\fn(E[p]/\bQ_\infty)$ is finite if and only if $\Sel^\Sigma_\fn(E/\bQ_\infty)[p]$ is finite. Since $\Sel^\Sigma_\fn(E/\bQ_\infty)$ (hence $\Sel^\Sigma_\fn(E[p]/\bQ_\infty)$) is cotorsion, the above finiteness results are equivalent to the vanishing of the corresponding $\mu$-invariants from the structure theorem. Finally, the vanishing of $\mu(\Sel_\fn^\Sigma(E/\bQ_\infty)^\vee)$ is equivalent to the vanishing of $\mu(\Sel_\fn(E/\bQ_\infty)^\vee)$ by~\cref{Sigma part mu=0} (see also~\cref{fine Selmer mu=0}).
\end{proof}

\begin{remark}\label[remark]{imprimK}
    \begin{enumerate}
        \item As in~\cref{finemuK}, if one replaces $\bQ_\infty$ by $K_\infty$ in the second row, and replaces $\H^i(\bQ_{\infty,p},-)$ by $\prod_{w\mid p}\H^i(K_{\infty,w},-)$ in the third row of the above diagram\, the same proof shows that $\mu(\Sel_\fn(E/K_\infty)^\vee)=0\Leftrightarrow \mu(\Sel^\Sigma_\fn(E[p]/K_\infty)^\vee)=0$ for a general abelian extension $K$ of $\bQ$, where the Selmer groups are known to be $\Lambda$-cotorsion.
        \item The exact same proof also shows $\mu(A^\fn_{K,\infty})=0\Leftrightarrow \mu(A[p]^{\fn,\Sigma}_{K,\infty})=0$ by replacing $E[p^\infty]$ with $\bQ_p/\bZ_p$, and $E[p]$ with $\F(\mathbf{1})$, taking into account (i). This works as long as $\Sigma$ contains the places above $p$.
    \end{enumerate} 
\end{remark}

\begin{proposition}\label[proposition]{Grmumodp}
$\mu(\Sel_\Gr(E/\bQ_\infty)^\vee)=0\Leftrightarrow \mu(\Sel^\Sigma_\Gr(E[p]/\bQ_\infty)^\vee)=0.$
\end{proposition}
\begin{proof}
    Consider the diagram{\footnotesize
	\begin{equation*}\begin{tikzcd}
			0 \ar[r]& \ker(\loc^\Gr_{/p}) \ar[r]\ar[d]&\Sel^\Sigma_\Gr(E[p]/\bQ_\infty)  \ar[r]\ar[d]&  \Sel^\Sigma_\Gr(E/\bQ_\infty)[p] \ar[d] &\\
			0 \ar[r]&  \H^0(\bQ^\Sigma/\bQ_\infty,E[p^\infty])/p\ar[r]\ar[d,"\loc^\Gr_{/p}"]& \H^1(\bQ^\Sigma/\bQ_\infty,E[p])\ar[r]\ar[d]& \H^1(\bQ^\Sigma/\bQ_\infty,E[p^\infty])[p]
            \ar[r]\ar[d]&0\\
			0 \ar[r]& \H^0(I_{\infty,p},\Fil^-(E[p^\infty]))/p \ar[r]\ar[d]& \H^1(I_{\infty,p},\Fil^-(E[p]))\ar[r]& \H^1(I_{\infty,p},\Fil^-(E[p^\infty]))[p]\ar[r]& 0 \\
            &\coker(\loc^\Gr_{/p})& &&
\end{tikzcd}\end{equation*}}
coming from the sequence~\eqref{E[p^infty]} and the sequence\begin{equation*}
    0\to \Fil^-(E[p]) \to \Fil^-(E[p^\infty]) \xrightarrow{p} \Fil^-(E[p^\infty])\to 0,\end{equation*}
    where $\ker(\loc^\Gr_{/p})$ and $\coker(\loc^\Gr_{/p})$ are finite since the domain and codomain of $\loc^\Gr_{/p}$ are finite. The rest of the proof is then identical to that of~\cref{finemumodp}.
\end{proof}
\begin{remark}\label[remark]{imprimclassical}
    The same proof also shows that $\mu(A_{K,\infty})=0\Leftrightarrow \mu(A[p]^\Sigma_{K,\infty})=0$ by replacing $E[p^\infty]$ and $\Fil^-(E[p^\infty])$ with $\bQ_p/\bZ_p$, and $E[p]$ and $\Fil^-(E[p])$ with $\F(\mathbf{1})$, taking into account~\cref{imprimK}(i). This holds for an arbitrary number field $K$ since $A_{K,\infty}$ is $\Lambda$-torsion for any $K$ from Iwasawa's theorem. Again, this holds as long as $\Sigma$ contains places above $p$.
\end{remark}

\subsubsection{Vanishing of fine $\mu$-invariants over $\bQ(E[p])$}\label{vanishingQ(E[p])}
In this subsubsection, we prove the vanishing of several $\mu$-invariants associated to $\bQ(E[p])$ (\cref{classical}). Namely, we will prove the classical $\mu(A_{\bQ(E[p]),\infty})$, the classical fine $\mu(A^\fn_{\bQ(E[p]),\infty})$ and the fine $\mu(\Sel_\fn(E/\bQ(E[p])_\infty)^\vee)$ for $E$ are all $0$. The proof is divided into two steps:\begin{enumerate}
    \item Vanishing of the $\mu$-invariants over $\bQ(\phi,\psi)$;
    \item Vanishing of the $\mu$-invariants over $\bQ(E[p])$.
\end{enumerate}
Here in step (i), $\phi,\psi$ fit into a non-split short exact sequence\[
0\to \F(\phi) \to E[p] \to \F(\psi) \to 0,
\]
where $\phi$ is unramified and even, and $\psi$ is ramified and odd. Note that if $p=3$, $\phi$ is forced to be $\mathbf{1}$ and $\psi$ is necessarily $\omega$.\\

\noindent \emph{Proof of step (i): }The vanishing of classical (and hence classical fine) $\mu$ follows directly from Ferrero--Washington theorem.
The vanishing of fine $\mu$ for $E$ then further follows from~\cite[Corollary 3.6]{CS05}, as $\phi$ becomes trivial (so $E[p]$ admits a rational point) over $K\defeq \bQ(\phi,\psi)$, which is an abelian extension of $\bQ$. In particular, from~\cref{imprimK} and~\cref{imprimclassical}, we have\begin{itemize}
    \item $\mu(\Sel^\Sigma_\fn(E[p]/K_\infty)^\vee)=\mu(\Sel_\fn(E/K_\infty)^\vee)=0$;
    \item $\mu(A[p]^{\fn,\Sigma}_{K,\infty})=\mu(A^\fn_{K,\infty})=0$;
    \item $\mu(A[p]^\Sigma_{K,\infty})=\mu(A_{K,\infty})=0$.
\end{itemize}

\noindent \emph{Proof of step (ii): }This follows from the following `$p$-extension principle', which also appears in Iwasawa's Riemman--Hurwitz formula ($\mu$-part only). Here we give a proof using Nakayama's lemma. We first prove it for the fine Selmer group for $E$, and the other two cases follow similarly.

\begin{proposition}
    Let $L/K$ be a $p$-extension. If $\mu(\Sel_\fn^\Sigma(E[p]/K_\infty)^\vee)=0$, then $\mu(\Sel_\fn^\Sigma(E[p]/L_\infty)^\vee)=0$
\end{proposition}
\begin{proof}
    Further assume that $\Sigma$ contains places of $K$ that are ramified in $L$. Let $G=\Gal(L/K)\iso \bZ/p^n\bZ$. Then one knows the restriction map\[
    \Sel_\fn^\Sigma(E[p]/K_\infty)\xrightarrow{\res_\fn} \Sel_\fn^\Sigma(E[p]/L_\infty)^G
    \]
    has finite kernel and cokernel. Indeed, the middle vertical map in the diagram{\footnotesize
	\begin{equation*}\begin{tikzcd}
			0 \ar[r]& \ker(\res_\fn) \ar[r]\ar[d]&\ker(\res)  \ar[r]\ar[d]&  \ker(\res_p)\ar[d] &\\
			0 \ar[r]&  \Sel_\fn^\Sigma(E[p]/K_\infty)\ar[r]\ar[d,"\res_\fn"]& \H^1(K^\Sigma/K_\infty,E[p])\ar[r]\ar[d,"\res"]& \im(\H^1(K^\Sigma/K_\infty,E[p]))
            \ar[r]\ar[d,"\res_p"]&0\\
			0 \ar[r]& \Sel_\fn^\Sigma(E[p]/L_\infty)^G \ar[r]\ar[d]& \H^1(L^\Sigma/L_\infty,E[p])^G\ar[r]\ar[d]& \prod_{l\mid w\mid p}\H^1(L_{\infty,l},E[p])& \\
            &\coker(\res_\fn)\ar[r]& \coker(\res) &&
\end{tikzcd}\end{equation*}}
has finite kernel and cokernel, since from the inflation-restriction exact sequence, $\ker(\res)=\H^1(G,E[p]^{G_{L_\infty}})$ and $\coker(\res)$ is contained in $\H^2(G,E[p]^{G_{L_\infty}})$, which are both finite as $G$ and $E[p]^{G_{L_\infty}}$ are both finite. Here $\im(\H^1(K^\Sigma/K_\infty,E[p]))$ detnoes $\im(\H^1(K^\Sigma/K_\infty,E[p])\to \prod_{w\mid p}\H^1(K_{\infty,w},E[p]))$.

Similarly, the local restriction\[
\prod_{w\mid p}\H^1(K_{\infty,w},E[p])\xrightarrow{\res_p}\prod_{l\mid w\mid p}\H^1(L_{\infty,l},E[p])
\]
has finite kernel, so the $\ker(\res_p)$ in the diagram is finite. Thus both $\ker(\res_\fn)$ and $\coker(\res_\fn)$ are finite from the snake lemma. Now if $\mu(\Sel_\fn^\Sigma(E[p]/K_\infty)^\vee)=0$, then $\mu((\Sel_\fn^\Sigma(E[p]/L_\infty)^G)^\vee)=0$ as well, and $(\Sel_\fn^\Sigma(E[p]/L_\infty)^G)^\vee$ is a finitely generated $\bZ_p$-module.

    Now let $A=\bZ_p[G]$ with (unique) maximal ideal $\fm$, and let $I_G\defeq \ker(A\to \bZ_p)$ be the augmentation ideal of $A$. Then $I_G\subset \fm$ and $\bZ_p=A/I_G$. We now know that $(\Sel_\fn^\Sigma(E[p]/L_\infty)^G)^\vee=(\Sel_\fn^\Sigma(E[p]/L_\infty))^\vee_G=(\Sel_\fn^\Sigma(E[p]/L_\infty))^\vee/I_G$ is a finitely generated $A/I_G$-module, so Nakayama's lemma implies that $(\Sel_\fn^\Sigma(E[p]/L_\infty))^\vee$ is finitely generated as a $A$-module, hence it is also finitely generated as a $\bZ_p$-module. That is, $\mu((\Sel_\fn^\Sigma(E[p]/L_\infty))^\vee)=0.$
\end{proof}

Since $\bQ(E[p])/\bQ(\phi,\psi)$ is either a $p$-extension or trivial from~\cref{Q(E[p])}, we can choose $L=\bQ(E[p])$ and $K=\bQ(\phi,\psi)$. Then we can apply the above proposition together with step (i) and~\cref{imprimK},~\cref{imprimclassical} to conclude step (ii), that is, $\mu(\Sel_\fn(E/L_\infty)^\vee)=\mu(\Sel_\fn^\Sigma(E[p]/L_\infty)^\vee)=0$. Finally, the above argument still works if one replaces $E[p]$ by $\F(\mathbf{1})$, and decomposition groups by inertia groups, so $\mu(A^\fn_{L,\infty})=\mu(A[p]^{\fn,\Sigma}_{L,\infty})=0$ and $\mu(A_{L,\infty})=\mu(A[p]^\Sigma_{L,\infty})=0$ as well.

\subsubsection{Proof of the main theorem}\label{proofmain}

We are now ready to prove~\cref{mainthm} by going to the Galois extension $L\defeq \bQ(E[p])$, where $E[p]$ splits as $\F(\phi)\oplus\F(\psi)$. We enlarge $\Sigma$ so that it also contains primes ramified in $L$. We assume that $\psi$ is ramified and odd. As was discussed in the introduction, we will study the commutative diagram\begin{center}{\small\begin{equation*}\label{gotoQ(E[p])}\tag{$\bQ$ to $L$}\begin{tikzcd}	
			 \H^1(\bQ^\Sigma/\bQ_\infty,E[p]) \ar[r,"\res_{L/\bQ}"]\ar[d,"loc_\bQ"]& \H^1(L^\Sigma/L_\infty,E[p])\ar[d,"loc_{L}"]\\
			 \H^1(\bQ_{\infty,p},E[p])\ar[r,"\prod_{w_i\mid p}\res"]&\prod_{w_i\mid p}\H^1(L_{\infty,w_i},E[p]),\\	 	\end{tikzcd}
             \end{equation*}}
\end{center}

\noindent We shall in fact consider an extended diagram (note that the rows are \textit{not} exact):\begin{center}{\small\begin{equation*}\begin{tikzcd}	
			 \H^1(\bQ^\Sigma/\bQ_\infty,E[p]) \ar[r,"\res_{L/\bQ}"]\ar[d,"loc_\bQ"]& \H^1(L^\Sigma/L_\infty,E[p])\ar[d,"loc_{L}"]\ar[r,"\proj_L"] &\H^1(L^\Sigma/L_\infty,\F(\phi))\ar[d,"\loc_{L,\F(\phi)}"]\\
			 \H^1(\bQ_{\infty,p},E[p])\ar[r,"\prod_{w_i\mid p}\res"]&\prod_{w_i\mid p}\H^1(L_{\infty,w_i},E[p])\ar[r,"\prod \proj_i"]&\prod_{w_i\mid p}\H^1(L_{\infty,w_i},\F(\phi)),\\	 	\end{tikzcd}
             \end{equation*}}
\end{center}where projection refers to taking the quotient by the ramified part $\F(\psi)$.

Recall that our goal is to show the `$\mu=1$' part of $\H^1(\bQ^\Sigma/\bQ_\infty,E[p])$ doesn't fully land in $\H^1(\bQ_{\infty,p},\F(\psi))$ under $\loc_\bQ$. By the commutativity of the diagram~\eqref{gotoQ(E[p])}, it suffices to show the `$\mu=1$' part of $\H^1(\bQ^\Sigma/\bQ_\infty,E[p])$ doesn't fully land in $\prod_{w_i\mid p}\H^1(L_{\infty,w_i},\F(\psi))$ under $\loc_L\circ\res_{L/\bQ}$. Thus by the commutativity of the extended diagram, we need to show the `$\mu=1$' part survives in $\prod_{w_i\mid p}\H^1(L_{\infty,w_i},\F(\phi))$, or equivalently, $\ker(\loc_{L,\F(\phi)}\circ \proj_L\circ \res_{L/\bQ})$ is finite.

The proof of our main result then consists of two steps:\begin{enumerate}
    \item[\emph{Step I:}] 
$\ker(\proj_L\circ\res_{L/\bQ})$ is finite;
    \item[\emph{Step II:}] 
    $\ker(\loc_{L,\F(\phi)})$ is finite.
\end{enumerate}
The group structures then force $\ker(\loc_{L,\F(\phi)}\circ \proj_L\circ \res_{L/\bQ})$ to be finite as well.\\

\noindent \emph{Proof of step I: }This is where we use the non-splitness of~\eqref{seq:E[p]}. Since the sequence\[
0\to \H^0(L^\Sigma/L_\infty,\F(\psi))\to\H^0(L^\Sigma/L_\infty,E[p])\to\H^0(L^\Sigma/L_\infty,\F(\phi))\to 0
\]is exact (using spliness of $E[p]$ over $L$, as the actions are all trivial and they obtain maximal invariants of sizes $p$, $p^2$ and $p$ respectively), $\ker(\proj_L)$ is given by $\H^1(L^\Sigma/L_\infty,\F(\psi))$. The finiteness of $\ker(\proj_L\circ\res_{L/\bQ})$ is then equivalent to the finiteness of $\im(\res_{L/\bQ})\cap \H^1(L^\Sigma/L_\infty,\F(\psi))$.

Note that $\im(\res_{L/\bQ})$ lies in $\H^1(L^\Sigma/L_\infty,E[p])^{\Gal(L_\infty/\bQ_\infty)}$. If the intersection $\im(\res_{L/\bQ})\cap \H^1(L^\Sigma/L_\infty,\F(\psi))$ was infinite, $\H^1(L^\Sigma/L_\infty, \F(\psi))$ would admit an infinite $\Gal(L_\infty/\bQ_\infty)$-invariant subspace. We will show that this is not possible.

Since~\eqref{seq:E[p]} is non-split over $G_\bQ$ and $G_L$ acts trivially on $E[p]$,~\eqref{seq:E[p]} remains non-split over $\Gal(L/\bQ)$ and $\F(\psi)$ is not a $\Gal(L/\bQ)$-invariant submodule. In other words, if we denote by $\rho_E:\Gal(L/\bQ)\to \Aut(E[p])$ the Galois action on $E[p]$ and assume $P$ is a generator of $\F(\psi)$, then there exists $\sigma\in \Gal(L/\bQ)$ such that $\rho_E(\sigma)(P)\notin \F(\psi)$. In the following, we identify $\Gal(L_\infty/\bQ_\infty)\iso \Gal(L/\bQ)$ since $L\cap\bQ_\infty=\bQ$. Indeed, if not, by a field degree comparison we see that $L\cap\bQ_\infty$ is at most (and will be) $\bQ_1$, the first layer in $\bQ_\infty$. But then $\bQ(E[p])$ must be a strict $p$-extension of $\bQ(\phi,\psi)$ while being equal to $\bQ(\phi,\psi)\bQ_1$. The former is not abelian over $\bQ$ but the latter is. Contradiction arises.

Now assume by contradiction that the intersection of the image of $\H^1(\bQ^\Sigma/\bQ_\infty, E[p])$ in $\H^1(L^\Sigma/L_\infty, E[p])$ and $\H^1(L^\Sigma/L_\infty, \F(\psi))$ was infinite. Then there exists a nontrivial $ f\in \H^1(L^\Sigma/L_\infty,\F(\psi))=\Hom(\Gal(L^\Sigma/L_\infty), \F(\psi))$ that is $\Gal(L_\infty/\bQ_\infty)=\Gal(L/\bQ)$-invariant. Since $f$ is nontrivial, there exists $\tau\in \Gal(L^\Sigma/L_\infty)$ such that $f(\tau)=P$. However, $(\sigma\cdot f )(\tau)=\sigma(f(\sigma^{-1}\tau \sigma))=\sigma(f(\tau))=\sigma(P)\notin \F(\psi)$, contradicting the invariance $\sigma\cdot f=f$. Therefore, the image must be finite, and in fact, trivial.

Finally, the finiteness of $\ker(\res_{L/\bQ})$ follows from the inflation restriction exact sequence, as both $[L:\bQ]$ and the coefficient $E[p]$ are finite.

\noindent \emph{Proof of step II: }
From the definitions, $\ker(\loc_{L,\F(\phi)})^\vee$ is nothing but $A[p]^{\fn,\Sigma}_{L,\infty}$. Thus the result follows from~\cref{Ferrero-Washington} and~\cref{imprimK}.

\subsection{The Iwasawa Main Conjecture}\label{IMC}
\begin{theorem}[Greenberg--Vatsal, Castella--Grossi--Skinner, Keller--Yin]
    Let $E/\bQ$ be an elliptic curve and let $p$ be an odd prime of good ordinary reduction for $E$. Assume that $E$ admits a rational $p$-isogeny, i.e., $p$ is an \emph{Eisenstein} prime for $E$. Then the Iwasawa Main Conjecture holds for $E$ at $p$, i.e.,\[
    \Char_\Lambda(\fX_{\Gr}(E/\bQ_\infty))=(\cL_E^\MSD)\text{ in }\Lambda.
    \]
\end{theorem}
\begin{proof}
    The theorem was first proved in~\cite{GV00} in case (i) of~\eqref{seq:E[p]in2}. Using a different method,~\cite{CGS} proved it in all cases assuming $\phi|_{G_p}\ne\mathbf{1}$ or $\omega$. The last hypothesis was removed in~\cite{KY24}.
\end{proof}

The following corollary is immediate from the above Iwasawa Main Conjecture and the definitions.
\begin{corollary}\label[corollary]{alg=anmu}
    Let $E/\bQ$ be an elliptic curve and let $p$ be an odd prime of good ordinary reduction for $E$. Assume that $E$ admits a rational $p$-isogeny, i.e., $p$ is an \emph{Eisenstein} prime for $E$. Then \[
    \mu^\alg(E)=\mu^\an(E),
    \]
    that is, the algebraic $\mu$-invariant of $E$ agrees with its analytic $\mu$-invariant.
\end{corollary}

\subsection{The vanishing of analytic $\mu$-invariants}
\begin{theorem}
    Let $E/\bQ$ be an elliptic curve and let $p$ be an odd Eisenstein prime of good ordinary reduction for $E$. Let $E_\min$ be the minimal curve in the isogeny class of $E$ in the sense of Stevens~\cite{Stevens}. Then $\mu^\an(E_\min)=\mu^\alg(E_\min)=0$.
\end{theorem}
\begin{proof}
    This is a combination of~\cref{mainthm},~\cref{alg=anmu} and Schneider's formula (namely~\cref{Schformula} and~\cite[Proposition (4.12)]{Stevens}).
\end{proof}

\section{The residually irreducible case}\label{sec:4}
Our proof of the main theorem shows that even in the residual irreducible case, one can still pass from the vanishing of classical fine $\mu$-invariants of certain extension field to the vanishing of Greenberg's $\mu$-invariant over $\bQ$. The following is~\cref{irred} in the introduction. Note that a hidden assumption is the $\Lambda_L$-torsionness of $A^\fn_{L,\infty}$, which is known for cyclotomic $\bZ_p$-extensions.
\begin{theorem}
     Let $E$ be an elliptic curve and $p$ be an odd prime of ordinary reduction for $E$. Assume $E[p]$ is irreducible as a $G_\bQ$-module. Let $L=\bQ(E[p])$. Enlarge $\Sigma$ to include primes ramified in $L$. Then $\mu(A^\fn_{L,\infty})=0\Rightarrow$ $\Sel_\fn(E/L_\infty)$ is $\Lambda_L$-cotorsion with $\mu=0$, and $\mu(\Sel_\Gr(E/\bQ_\infty)^\vee)=0.$
\end{theorem}
\begin{proof}
    Let's consider again the extended diagram with\begin{center}{\small\begin{equation*}\begin{tikzcd}	
			 \H^1(\bQ^\Sigma/\bQ_\infty,E[p]) \ar[r,"\res_{L/\bQ}"]\ar[d,"loc_\bQ"]& \H^1(L^\Sigma/L_\infty,E[p])\ar[d,"loc_{L}"]\ar[r,"\proj_L"] &\H^1(L^\Sigma/L_\infty,\F(\phi))\ar[d,"\loc_{L,\F(\phi)}"]\\
			 \H^1(\bQ_{\infty,p},E[p])\ar[r,"\prod_{w_i\mid p}\res"]&\prod_{w_i\mid p}\H^1(L_{\infty,w_i},E[p])\ar[r,"\prod \proj_i"]&\prod_{w_i\mid p}\H^1(L_{\infty,w_i},\F(\phi)),\\	 	\end{tikzcd}
             \end{equation*}}
\end{center}
where the splitting of $E[p]$ over $L$ remains true. In this case, globally $\F(\psi)$ is defined as the fixed $\F$-line in $E[p]$ that carries the $G_{\bQ_p}$ action of $\phi$, which is not a $G_\bQ$-submodule by irreducibility, which becomes a $G_L$-submodule of $E[p]$, with a quotient denoted by $\F(\psi)$. The splitting of $E[p]$ over $\bQ_p$ was never used, and merely that $\F(\psi)$ is a sub was necessary. 

First, $\mu(A^\fn_{L,\infty})=0$ implies that $\Sel_\fn^\Sigma(E/L_\infty)$ (and hence $\Sel_\fn(E/L_\infty)$) is $\Lambda_L$-cotorsion, and that $\mu(\Sel_\fn(E/L_\infty)^\vee)=\mu(\Sel^\Sigma_\fn(E[p]/L_\infty)^\vee)$ is $0$, for any $\Sigma$ containing places above $p$ and those where $E$ has bad reduction. Indeed, from the following diagram (everything trivializes over $L$){\footnotesize
	\begin{equation*}\begin{tikzcd}
			0 \ar[r]& \Sel_\fn^\Sigma(\F(\mathbf{1})/L_\infty) \ar[r]\ar[d]&\Sel^\Sigma_\fn(E[p]/L_\infty)  \ar[r]\ar[d]&  \Sel^\Sigma_\fn(\F(\mathbf{1})/L_\infty) \ar[d] &\\
			0 \ar[r]&  \H^1(L^\Sigma/L_\infty,\F(\mathbf{1}))\ar[r]\ar[d,"\loc^\Gr_{/p}"]& \H^1(L^\Sigma/L_\infty,E[p])\ar[r]\ar[d]& \H^1(L^\Sigma/L_\infty,\F(\mathbf{1}))
            \ar[r]\ar[d]&0\\
			0 \ar[r]& \prod_{w\mid p}\H^1(L_{\infty,w},\F(\mathbf{1})) \ar[r]& \prod_{w\mid p}\H^1(L_{\infty,w},E[p])\ar[r]& \prod_{w\mid p}\H^1(L_{\infty,w},\F(\mathbf{1}))\ar[r]& 0,
\end{tikzcd}\end{equation*}}
it is easy to see that in fact $\mu(A[p]^{\fn,\Sigma}_{L,\infty})=0$ if and only if $\mu(\Sel^\Sigma_\fn(E[p]/L_\infty)^\vee)=0$, but the former is equivalent to the assumption $\mu(A^\fn_{L,\infty})=0$, so in fact $\Sel^\Sigma_\fn(E[p]/L_\infty)$ is finite. The structure theorem then implies that $\Sel^\Sigma_\fn(E/L_\infty)$ is $\Lambda_L$-cotorsion with $\mu$-invariant $0$. So the same holds for $\Sel_\fn(E/L_\infty)$. Moreover, $\Sel^\Sigma_\fn(E[p]/\bQ_\infty)$ is finite as well. Thus $\mu(\Sel_\fn(E/\bQ_\infty)^\vee)=0$. What we need to show next is that no global $\mu$-factor lands fully in $\H^1(\bQ_\infty, \F(\psi))$. It suffices to show $\ker(\loc_{L,\F(\phi)}\circ \proj_L\circ \res_{L/\bQ})$ is finite. As before, it suffices to show $\ker(\loc_{L,\F(\phi)})$ is finite and $\ker(\proj_L\circ \res_{L/\bQ})$ is finite. The former follows from the assumption that $\mu(A^\fn_{L,\infty})=0$ (and that its torsion) by~\cref{imprimK}, and the latter is again equivalent to the finiteness of $\im(\res_{L/\bQ})\cap \H^1(L^\Sigma/L_\infty, \F(\psi))$.

Again, as in the non-split Eisenstein case, $\F(\psi)$ is not $\Gal(L/\bQ)$-invariant. The same argument as in~\cref{proofmain} shows that this kernel must indeed be finite, as $L=\bQ(E[p])$ and $\bQ_\infty$ are still linearly disjoint. 
\end{proof}

\bibliographystyle{amsalpha}
\bibliography{references}
\end{document}